\documentclass[10pt]{article}
\usepackage[T1]{fontenc}
\usepackage[cp1250]{inputenc}
\usepackage{lmodern}

\usepackage[english]{babel}
\usepackage{latexsym}
\usepackage{amsmath}
\usepackage{indentfirst}
\usepackage{amsthm}
\usepackage{amssymb}
\usepackage{amsfonts}
\usepackage{stackrel}
\usepackage{float}
\usepackage{dsfont}
\usepackage{tikz}
\usepackage{slashbox}
\usepackage[mathscr]{eucal}
\usepackage{authblk}
\usepackage{hyperref}
\usepackage{mathabx}
\usepackage{bookmark}
\usepackage{subcaption}
\usepackage[section]{placeins}

\newtheorem{thm}{Theorem}

\newtheorem{defin}{Definition}

\newcommand{\N}{\mathbb{N}}

\newcommand{\GG}{\mathbb{G}}
\newcommand{\EE}{\mathbb{E}}

\begin{document}
\title{Comparison of approximation algorithms for the travelling salesperson problem on semimetric graphs}
\author{Mateusz Krukowski, Filip Turobo\'s}
\affil{Institute of Mathematics, \L\'od\'z University of Technology, \\ W\'ol\-cza\'n\-ska 215, \
90-924 \ \L\'od\'z, \ Poland \\ \vspace{0.3cm} e-mail: mateusz.krukowski@p.lodz.pl}
\maketitle

\begin{abstract}
The aim of the paper is to compare different approximation algorithms for the travelling salesperson problem. We pick the most popular and widespread  methods known in the literature and contrast them with a novel approach (the polygonal Christofides algorithm) described in our previous work. The paper contains a brief summary of theory behind the algorithms and culminates in a series of numerical simulations (or ``experiments''), whose purpose is to determine ``the best'' approximation algorithm for the travelling salesperson problem on complete, weighted graphs.     
\end{abstract}

\vspace{0.2cm}

\section{Introduction}

The travelling salesperson problem (or the TSP for short) hardly needs any introduction. It is a sure bet that every mathematician at one point or another has heard something along those lines: ``Suppose we have $n$ cities and you are a salesperson delivering a product to each of those cities. You cannot travel the same road between two cities twice, and (obviously) you have to supply all the customers with the desired products (or else you get fired). At the end of the journey you have to come back to where you started (beacuse the boss is waiting for your report). In what order will you visit all the cities?''

Anyone who has attended at least a couple of lectures on graph theory will immediately recognize that the story is asking for a Hamiltonian path of minimal weight in a given graph $G$ (representing the cities and roads between them). However, the story as presented above demands that we deal with a few technical caveats. First off, we assume that $G$ is a complete and weighted graph. Intuitively, this means that between any two cities there is a direct road connecting them. 

Next, a couple of words regarding the weights of the edges (roads between cities) are in order. The first thought that springs to mind is that the ``weight of the road'' should be its length (in km, miles or whatever metric system is used in a given country). Under the assumption that the roads are as straight as a ruler, such a choice of weights leads to the so-called \textit{metric TSP}, where 
\begin{center}
    distance between city A and city B + distance between city B and city C \\
    $\geqslant$ distance between city A and city C.
\end{center}

However, anyone who owns a car knows fairly well that reality does not always pan out that conveniently. For instance, we may imagine a highway which goes straight from city A to city B and then to city C and that a direct road from city A to city C leads through hills and valleys and is filled with numerous turns. It is definitely conceivable that in such a scenario we actually have 
\begin{center}
    distance between city A and city B + distance between city B and city C \\
    $<$ distance between city A and city C.
\end{center}

If the Reader consider this example to be too far-fetched, let us suggest another reasoning.\footnote{For this argument to work we may even assume that the salesperson lives in a world where all roads are straight lines (or rather intervals).} As we all know from personal experience, time is a much more valuable resource in our lives than the number of kilometers (miles etc.) we have travelled. Hence, we should not be surprised that the salesperson would rather take the ring road, travel a longer distance but save precious time rather than get stuck on a shorter road in traffic jams at every junction. This means that if the weight of the edge/road is the time it takes to travel that distance, the TSP may easily be nonmetric. 

The discussion we carried out above supports the claim that the nonmetric instances of the travelling salesperson problem should not be discarded as ``uninteresting''. As we have argued, these instances model the scenarios we encounter in our daily lives and as such constitute a sufficient motivation for further research in this area.

Having justified why we feel that the nonmetric TSP is an essential part of mathematical research we proceed with laying out the general schedule of the paper. We present this brief overview to facilitate the comprehension of the ``big picture'' before we dive into technical details.

Section \ref{section:framework} introduces preliminary notions in graph theory and semimetric spaces, which are indispensible for further reading. Additionally, the section establishes the notation used throughout the paper. Section \ref{section:approximatesolutions} opens with an explanation of why the travelling salesperson problem is not as easy as ``simply checking all the Hamiltonian cycles'' on a graph. Although such an approach seems perfectly valid from a theoretical standpoint, the computational complexity of the TSP is so immense (even for relatively small graphs) that no computer will ever be able to ``brute-force this problem'' in reasonable time. Section \ref{section:approximatesolutions} goes on to describe the following methods, which return approximate solutions to the TSP:
\begin{description}
    \item[$\bullet$] double minimal spanning tree algorithm (or DMST algorithm for short),
    \item[$\bullet$] (refined) Andreae-Bandelt algorithm (or (r)AB algorithm for short),
    \item[$\bullet$] path matching Christofides algorithm (or PMCh algorithm for short),
    \item[$\bullet$] polygonal Christofides algorithm (or PCh algorithm for short), which is a novel method constructed in our previous paper.\footnote{See \cite{Krukowski2021}.}
\end{description}

\noindent
To every method we attach a pseudocode, so everyone (if they so please) can implement these algorithms in a programming language of their own choosing.

Section \ref{section:numericalcomparison} is where we put our new PCh algorithm to the test and juxtapose it with other methods generating approximate solutions to the TSP. We verify that the PCh method performs better (returns Hamiltonian cycles with lower total weight) than the rest of the algorithms on a series of random graphs of different sizes. We also test numerically that the execution time of the PCh algorithm does not deviate much from those of other algorithms (bar the DMST method, which is significantly faster). Naturally, the paper concludes with the bibliography.

\section{Framework of semimetric spaces and graph theory}
\label{section:framework}

In the introductory section we laid down (in rather broad terms) the travelling salesperson problem and argued that its nonmetric instances are equally important as their metric counterparts. It is high time we recalled the preliminaries of semimetric spaces in greater detail. 

\begin{defin}
For a nonempty, finite set $X$, a function $d:X\times X \longrightarrow [0,+\infty)$ is called a semimetric if it satisfies the following two conditions:
\begin{description}
	\item[$\bullet$] $\forall_{x,y\in X}\ d(x,y) = 0$ if and only if $x=y$, and
	\item[$\bullet$] $\forall_{x,y\in X}\ d(x,y) = d(y,x)$.
\end{description}

\noindent
The pair $(X,d)$ is called a semimetric space.
\label{semimetricdefinition}
\end{defin}		

Let us remark that we insist on set $X$ being finite simply because our model example and primary motivation is the travelling salesperson problem, which makes no sense on infinite graphs. Hence, we refrain from unnecessary and excessive generality and do not consider infinite semimetric spaces.

Next, we define $\beta-$metric spaces and $\gamma-$polygon spaces:\footnote{Both $\beta-$metric and $\gamma-$polygon spaces are well-established  in the literature: \cite{An2015,Andreae1995,Bakthin1989,Bandelt1991,Bourbaki1966,Chrzaszcz2018,Chrzaszcz2018.2,Dung2016,Fagin2003,Jachymski2020,Paluszynski2009,Schroeder2006,Suzuki2017,Wilson1931,Xia2009} serve just as a couple of examples.}

\begin{defin}
A semimetric space $(X,d)$ is said to be:
\begin{description}
    \item[$\bullet$] $\beta-$metric space if $\beta\geqslant 1$ is the smallest number such that the semimetric $d$ satisfies the $\beta-$triangle inequality:
    \begin{gather}
    \forall_{x,y,z\in X}\ d(x,z)\leqslant \beta ( d(x,y) + d(y,z) ),
    \label{betainequality}
    \end{gather}
    
    \item[$\bullet$] $\gamma-$polygon space if $\gamma\geqslant 1$ is the smallest number such that the semimetric $d$ satisfies the $\gamma-$polygon inequality:
    \begin{gather}
    \forall_{n\in\N} \ \forall_{x_1,\dots,x_n\in X}\ d(x_1,x_n)\leqslant \gamma\cdot \sum_{k=1}^{n-1} d(x_k,x_{k+1}).
    \label{polygonalinequality}
    \end{gather}
\end{description}
\end{defin}

It is a relatively easy observation\footnote{As far as we know it first appeared in \cite{Chrzaszcz2018.2}.} that every semimetric space admits both $\beta$-metric and $\gamma$-polygon structure. Naturally, $\beta\leqslant \gamma$ but the two constants may differ in general.\footnote{For an example see \cite{Krukowski2021}.} 

Let us proceed with a brief summary of graph theory notions which are necessary for further reading. An additional advantage of our concise review is that we lay down the notational conventions used in the sequel. There could be no other starting point than the definition of a graph itself:\footnote{The definition of a graph is based on \cite[p. 2]{Diestel2000}, whereas the definition of a weighted graph was taken from \cite[p. 463]{Fletcher1991}.}

\begin{defin}
A pair $G := (V,E)$ is called a graph if $V$ is a nonempty, finite set and 
$$E \subset \bigg\{\{x,y\} \ : \ x,y\in V,\ x\neq y \bigg\}.$$ 

\noindent
The elements of $V$ and $E$ are called vertices (or nodes) and edges, respectively.

A weighted graph is a pair $(G,\omega)$, where $G = (V,E)$ is a graph, $E\neq \emptyset$ and $\omega : E \longrightarrow (0,+\infty)$ is a positive function, called the weight.
\end{defin}

Let us pause for a moment and discuss this definition. First off, we will often utter the phrase ``\textit{Let $G$ be a graph}'' and then refer to the vertex and edge set of $G$ as $V(G)$ and $E(G)$, respectively.\footnote{An identical convention can be found in \cite{Diestel2000}.} Furthermore, those acquainted with the graph terminology will surely recognize our graphs to be \textit{undirected} and \textit{simple}. ``\textit{Undirectedness}'' of a graph means that the edges are not oriented, i.e., every edge is a set $\{x,y\}$ rather than an ordered pair $(x,y).$ On the other hand, ``\textit{simplicity}'' means that the graph contains no ``\textit{loops}'' (i.e. edges of the form $\{x,x\}$) or \textit{multiedges} (i.e. $E$ is a set and not a multiset). However, the need for \textit{multiedges} and \textit{multigraphs} will arise in Section \ref{section:approximatesolutions}, so we take the liberty of including the formal definition of these objects here:

\begin{defin}
A pair $\GG :=(V,\EE)$ is called a multigraph if $V$ is a nonempty, finite set and
$$\EE\subset \bigg\{ (k,\{x,y\})\ : \ x,y\in V,\ x\neq y,\ k\in\N_0 \bigg\}.$$

The elements of $\EE$ are called multiedges (while the elements of $V$ are still called vertices or nodes). A weighted multigraph is a pair $(\GG,\omega)$, where $\GG= (V,\EE)$ is a multigraph, $\EE\neq \emptyset$ and $\omega:\EE\longrightarrow (0,+\infty)$ is a positive function, called the weight.
\end{defin}

We proceed with a series of familiar graph theory concepts:\footnote{See \cite{Anderson2018,Bondy2008,Diestel2000}.}

\begin{defin}\label{def:subgraph etc}
Let $G$ be a graph.
\begin{description}
    \item[$\bullet$] Graph $F$ is called a subgraph of $G$ if $V(F)\subset V(G)$ and $E(F)\subset E(G)$. We denote this situation by writing $F\subset G$.
    
    \item[$\bullet$] Graph $P$ is called a path if $V(P)$ can be arranged in a sequence so that two vertices are adjacent if and only if they are consecutive in this sequence.
    
    \item[$\bullet$] Graph $C$ is called a cycle if the removal of any edge in $C$ turns it into a path. If $G$ does not contain any cycle as a subgraph, then it is called an acyclic graph.
    
    \item[$\bullet$] Subgraph $H\subset G$ is called a Hamiltonian cycle if it is a cycle visiting each vertex of $G$ exactly once. 
    
    \item[$\bullet$] Graph $G$ is said to be connected if for any pair of vertices $x,y\in V(G)$ there exists a path $P\subset G$ such that $x,y\in V(P)$.

    \item[$\bullet$] An acyclic, connected graph is called a tree.
\end{description}
\end{defin}

The notion of a subgraph introduced in Definition \ref{def:subgraph etc} is inherently independent of the weight function (if such exists) on graph $G$. However, if $\omega$ is a weight on $G$ and $F$ is a subgraph of $G$, then $(F,\omega|_{E(F)})$ is a weighted graph and $\omega|_{E(F)}$ is called an \textit{induced weight}. Furthermore, it is often convenient to speak of a weight of a subgraph (or the whole graph itself), which we define as 
\begin{equation*}
\omega(F):=\sum_{e\in E(F)} \omega(e).
\end{equation*}

It is hard to deny that this definition is a slight abuse of notation -- after all, we use the same symbol ``$\omega$'' to weigh both edges and (sub)graphs. Formally this is a mistake since edges and (sub)graphs are objects from different ``categories'' -- an edge is an unordered pair of elements while a (sub)graph is a pair of vertex set and an edge set. However, we believe that such a small notational inconsistency should not lead to any kind of misapprehension.

We are now in position to formulate the \textit{travelling salesperson problem} (or TSP for short) in a formal manner:\footnote{It should be emphasized that there are multiple other ways to define this problem, for example as an integer programming problem with constraints on vertex degrees -- see \cite{Danzig1959,Laporte1992,Matai2010,Miller1960}.}\\
\vspace{0.01cm}\\
\fbox{%
\parbox{0.95\textwidth}{\large
\textit{For a complete, weighted graph $(K_n,\omega)$ find a Hamiltonian cycle with minimal weight.}
}
}
\\\vspace{0.01cm}\\

As we remarked earlier, TSP is one of the most famous mathematical puzzles\footnote{Timothy Lanzone even directed a movie ``\textit{Travelling Salesman}'', which premiered at the International House in Philadelphia on June 16, 2012. The thriller won multiple awards at Silicon Valley Film Festival and New York City International Film Festival the same year.} and a detailed account of its history is far beyond the scope of this paper. Instead, we recommend just a handful of sources -- for an indepth discussion on both history and possible attempts at solving this problem see: \cite{Applegate2006,Cook2012,Fleischmann1988,Fonlupt1992,Gutin2001,Laporte1992,Reinelt1994,Snyder2019,Van2020}.

To conclude this section let us bridge the gap between graph theory and semimetric spaces. Given a complete, weighted graph $G := (K_n, \omega)$ it is most natural to impose the semimetric structure on $V(K_n)$ in the following way:\footnote{Clearly, the function defined in \eqref{semimetriconwholegraph} is symmetric (due to the ``undirectedness'' of the graph $G$) and equals zero if and only if $x=y$ for all $x,y\in V(G)$.}
\begin{equation}\label{semimetriconwholegraph}
d_G(x,y):=\begin{cases} 
\omega(\{x,y\}),& \text{ if }x\neq y,\\
0,& \text{ otherwise.}
\end{cases}
\end{equation}

Due to this ``graph-semimetric marriage'' we are able to introduce the following convenient definition:

\begin{defin}
A complete, weighted graph $G:= (K_n,\omega)$ is called
\begin{description}
    \item[$\bullet$] a metric graph if $d_G$ is a metric on $V(K_n)$,
    \item[$\bullet$] a $\beta-$metric graph if $d_G$ satisfies \eqref{betainequality},
    \item[$\bullet$] a $\gamma-$metric graph if $d_G$ satisfies \eqref{polygonalinequality}.
\end{description}
\end{defin}

\section{Approximate solutions to the TSP on semimetric graphs}
\label{section:approximatesolutions}

A brute-force solution to the travelling salesperson problem is to go over all possible Hamiltonian cycles in a given graph, calculate the total weight of each of them and choose one with the smallest value.\footnote{Note that we do not say \textit{the} one with the smallest value, since there might be multiple Hamiltonian cycles with equal, minimal total weight.} Each Hamiltonian cycle can be thought of as a permutation of nodes -- for instance, the Hamiltonian cycle $(x_1,x_2,x_3,\ldots, x_{n-1},x_n)$ in $K_n$ corresponds to the permutation $(1,2,3,\ldots,n-1,n)$ whereas $(x_n,x_2,x_3,\ldots,x_{n-1},x_1)$ corresponds to $(n,2,3,\ldots,n-1,1).$ However, the permutation representation of a Hamiltonian cycle is not unique, since $(1,2,3,\ldots,n-1,n),\ (2,3,\ldots,n-1,n,1),\ (3,\ldots,n-1,n,1,2),\ \ldots,\ (n-1,n,1,2,3,\ldots,n-2)$ and $(n,1,2,3,\ldots,n-1)$ all represent the same Hamiltonian cycle, namely $(x_1,x_2,x_3,\ldots,x_{n-1},x_n).$ Furthermore, every ``\textit{order reversal}'' also represents the same cycle, so $(n,n-1,\ldots,3,2,1),\ (1,n,n-1,\ldots,3,2),\ (1,2,n,n-1,\ldots,3),\ \ldots$ and $(n-1,\ldots,3,2,1,n)$ also correspond to $(x_1,x_2,x_3,\ldots,x_{n-1},x_n).$ In summary, there are $2n$ permutations of the set $\{1,2,\ldots,n\},$ which correspond to a single Hamiltonian cycle. This means that in order to find a ``true'' (rather than an approximate) solution to the travelling salesperson problem, one needs to examine $\frac{1}{2}\cdot (n-1)!$ permutations.\footnote{There are a number of algorithms for generating all the permutations of a given set (for a thorough exposition see \cite{Sedgewick1977}), with one of the most popular being the Heap's algorithm (see \cite{Heap1963}).} 

To illustrate the computational complexity of the TSP let us imagine a complete, weighted graph $K_{61}.$ Due to the analysis above, there are $\frac{1}{2}\cdot 60!$ Hamiltonian cycles on this graph, which is more than the estimated number of particles in the observable universe! This means that even for relatively small graphs ($61$ nodes are certainly within human comprehension) the brute-force approach of checking every Hamiltonian cycle and looking for the one with the minimal total weight is extremely time-consuming. 

One possibility of speeding up the search for the ``ideal'' solution to the travelling salesperson problem is to use the Held-Karp algorithm.\footnote{See \cite{Held1971}.} This method relies on the function $HK$ defined recursively for every node $x \in \{x_2,\ldots,x_n\}$ and every subset of nodes $S\subset \{x_2,\ldots,x_n\}\backslash \{x\}$ as 
$$HK(x, S) := \min_{y\in S}\ \bigg(HK(y, S\backslash\{y\}) + d_G(y,x)\bigg)\hspace{0.4cm}\text{and}\hspace{0.4cm} HK(x,\emptyset) := d_G(x_1,x).$$

\noindent
Intuitively, the value $HK(x,S)$ is the total weight of the path between node $x_1$ and node $x,$ which goes through the vertices of $S$ in some order (the path does not use any other nodes). With this interpretation in mind, solving the TSP boils down to calculating all the values $HK(x_2,\{x_2,\ldots,x_n\}\backslash\{x_2\}),$ $HK(x_3,\{x_2,\ldots,x_n\}\backslash\{x_3\}),$ $\ldots,$ $HK(x_n,\{x_2,\ldots,x_n\}\backslash\{x_n\})$ and choosing the minimal value. Tracing back all the choices of the algorithm we are able to reconstruct the Hamiltonian cycle with the minimal weight (i.e., an ideal solution to the TSP). 

Although the Held-Karp algorithm is a reasonable improvement with respect to the brute-force approach it is still characterized by the exponential time complexity. This renders the method impractical for graphs of larger size. Hence the need for algorithms that solve the TSP considerably faster even at the price of returning approximate solutions rather than the ideal ones. The current section aims to summarize these approximation methods.

\subsection{Double minimal spanning tree algorithm}

As the name itself suggests, the first method we discuss hinges upon the notion of a (minimal) spanning tree:

\begin{defin}
\label{def:paths trees spanning trees}
Let $(K_n,\omega)$ be a complete, weighted graph. A tree $T$, which satisfies $V(T) = V(K_n)$ is called a spanning tree. Furthermore, if $\mathcal{S}$ denotes the set of all spanning trees of $(K_n,\omega)$, then a tree $T$ which satisfies
$$\omega(T) = \min_{T' \in \mathcal{S}} \omega(T').$$

\noindent
is called a minimal spanning tree. 
\end{defin}

We usually say that $T$ is ``\textit{a}'' rather than ``\textit{the}'' minimal spanning tree since for a given graph there might be more than one such tree. The most obvious example for such a situation is a complete graph with every edge of equal weight. 

Prior to elaborating on the the dobule minimal spanning tree method itself, let us recall the concepts of a walk and a tree traversal:\footnote{See \cite[p.10]{Diestel2000}.}

\begin{defin}
\label{def:walk and tree traversal}
Let $G$ be a graph (not necessarily complete or weighted). A $j$-element sequence of vertices $(x_1,\dots,x_j)$ such that $\{x_i,x_{i+1}\}\in E(G)$ for every $i<j$ is called a walk on graph. If $T$ is a tree on $K_n$, then any walk on $T$ which visits every edge exactly twice is called a tree traversal.
\end{defin}

We are now fully-equipped to formulate the general framework of the \textit{double minimal spanning tree algorithm} (or the \textit{DMST algorithm} for short):\\

\noindent
\textbf{Input:} A complete, weighted graph $(K_n,\omega).$
\begin{description}
    \item[\textbf{Step 1.}] Find a minimal spanning tree $T$ of the graph. 
    \item[\textbf{Step 2.}] Via a depth-first search (or DFS for short) on $T$ construct a tree traversal (which depends on the root of the algorithm). 
    \item[\textbf{Step 3.}] Perform a shortcutting procedure on the tree traversal (from the previous step) to obtain a Hamiltonian cycle $H_{DMST}$. 
\end{description} 

The first two steps are widely discussed in numerous sources.\footnote{See \cite[Chapter 3.10]{Deo1974} or \cite[Chapter 12]{Papadimitriou1998} for the algorithms generating a minimal spanning tree, and \cite[Chapter 22.3]{Cormen2009} for the DFS algorithm.} Hence, we restrict ourselves to presenting an overview of the the shortcutting procedure. 

Given a minimal spanning tree $T$ in a complete, weighted graph $(K_n,\omega)$ we choose an arbitrary vertex $x_1\in V(T),$ which we refer to as the \textit{root of the tree}. Performing a DFS yields a tree traversal $(x_1,\ldots, x_{2n-1}).$ Next, we define $(y_1,\dots,y_n,y_{n+1})$ to be a sequence obtained from the traversal $(x_1,\dots,x_{2n-1})$ by ``\textit{shortcutting}'', i.e.:
\begin{equation}\label{eq:definition_of_yi}
    y_i:=\begin{cases}
    x_1, & i=1 \mbox{ or } i=(n+1), \\
    x_{m(i)}, & 1< i <n+1,
    \end{cases}
\end{equation}

\noindent
where 
$$m(i):=\min\bigg\{j: i\leqslant j \leqslant 2n-1 \ \text{and} \ x_j \not\in \{x_1,\dots,x_{j-1} \} \bigg\}.$$

\noindent
Although the definition (\ref{eq:definition_of_yi}) seems rather daunting, there is in fact a simple way to obtain $(y_1,\dots,y_n,y_{n+1})$ from $(x_1,\ldots,x_{2n-1})$ -- we go through the elements of the tree traversal one by one and cross out every ``repetition'' (an element that we have seen earlier in the sequence) except for $x_{2n-1},$ which is the same as $x_1$. 

Shortcutting procedure results in creation of a sequence $(y_i)_{i=1}^{n+1},$ where every vertex (except for the root $x_1$) appears precisely once (after all, we did cross out all repetitions other than $x_1=x_{2n-1}$). This means that the sequence is in fact a Hamiltionian cycle on $K_n$, which we denote by $H_{DMST}$ and refer to as the \textit{DMST Hamiltonian cycle}. We should, however, bear in mind that there might be multiple DMST Hamiltonian cycles on a given graph, which depend on the choice of both the minimal spanning tree $T$ and the root of the algorithm $x_1.$

If $H_{ideal}$ denotes the ideal (i.e., not approximate) solution to the TSP, then every minimal spanning tree $T$ satisfies the inequality:
\begin{equation}\label{eq:comparing optimal and MST}
\omega(T)\leqslant \omega(H_{ideal}).
\end{equation}

\noindent
The proof of this fact can be shortened to a simple observation that removing a single edge from $H_{ideal}$ leaves us with a spanning tree, whose cost is (by the very definition of the minimal spanning tree) bounded from below by $\omega(T)$. Inequality \eqref{eq:comparing optimal and MST} enables us to write down the following:

\begin{thm}\label{thm:on approximation of TSP in polygon graphs}
Let 
\begin{description}
    \item[$\bullet$] $(K_n,\omega)$ be a complete, $\gamma$-polygon graph,
    \item[$\bullet$] $T$ be its minimal spanning tree,
    \item[$\bullet$] $x_1$ be any node,
    \item[$\bullet$] $H_{DMST}$ be a Hamiltonian cycle (corresponding to tree $T$ and root $x_1$) constructed by the double minimal spanning tree method.
\end{description} 

\noindent
Then
\begin{equation}
\label{omegaHDMST}
\omega(H_{DMST}) \leqslant 2\gamma \omega(H_{ideal}).
\end{equation}
\end{thm}
\begin{proof}
Let $(x_k)_{k=1}^{2n-1}$ be a tree traversal for the minimal spanning tree $T$ and let $H_{DMST} = (y_k)_{k=1}^{n+1}.$ We have
\begin{equation}
\forall_{i=1,\ldots,n}\ d(y_{i},y_{i+1}) = d(x_{m(i)},x_{m(i+1)}) \stackrel{ \eqref{polygonalinequality}}{\leqslant} \gamma \cdot \sum_{j=m(i)}^{m(i+1)-1} d(x_{j},x_{j+1}).
\label{rpi_case_assumption}
\end{equation}

\noindent
Consequently, we obtain
{\small \begin{gather*}
\omega(H_{DMST}) = \sum_{i=1}^{n} d(y_i,y_{i+1}) \stackrel{\eqref{rpi_case_assumption}}{\leqslant} 
\gamma \cdot \sum_{i=1}^n \sum_{j=m(i)}^{m(i+1)-1} d(x_{j},x_{j+1})
\leqslant \gamma\sum_{k=1}^{2n-2} d(x_k,x_{k+1}) \leqslant 2\gamma\omega(T).
\end{gather*}}

\noindent
Due to the inequality \eqref{eq:comparing optimal and MST} we conclude the proof. 
\end{proof}

\subsection{Andreae-Bandelt algorithms}

The basic version of the \textit{Andreae-Bandelt algorithm} (or \textit{AB algorithm} for short)  was laid down in \cite{Andreae1995}. The crux of the algorithm is the fact that the cube $T^3$ of any tree $T$ contains a Hamiltonian cycle.\footnote{Let us recall that for a tree $T$, the cube $T^3$ is a graph on the same set of vertices, i.e. $V(T^3) = V(T)$ and such that $(x,y) \in E(T^3)$ if and only if there exists a path in $T$ between $x$ and $y,$ which comprises of at most $3$ edges.} The Andreae-Bandelt method (referred to as the $T^3-$algorithm by the authors themselves) is summarized by the following pseudocode:\footnote{See \cite{Andreae2001}.}\\

\noindent
\textbf{Input:} A tree $T$ with $|V(T)|\geqslant 3$ and an edge $\{x_1,x_2\} \in E(T).$
\begin{description}
    \item[\textbf{Step 1.}] Let $T_i$ be the connected component of $T-\{x_1,x_2\}$ containing the node $x_i$ for $i=1,2.$ 
    \item[\textbf{Step 2.}] If $|V(T_i)|\geqslant 2$ pick any $y_i\in E(T_i)$ such that $\{x_i,y_i\} \in E(T_i).$ If $|V(T_i)| = 1$ put $y_i := x_i.$
    \item[\textbf{Step 3.}] If $|V(T_i)|\geqslant 3,$ apply recursively the algorithm with $T_i$ and $\{x_i,y_i\}$ as the inputs, thus obtaining a Hamiltonian cycle $H_i$ on $T_i^3$ which contains the edge $\{x_i,y_i\}.$
    \item[\textbf{Step 4.}] If $|V(T_i)| \geqslant 3$ put $P_i := H_i - \{x_i,y_i\},$ otherwise put $P_i := T_i.$
    \item[\textbf{Step 5.}] Construct the Hamiltonian cycle by joining $P_1, P_2$ and the edges $\{x_1,x_2\},\ \{y_1,y_2\}.$
\end{description}

Andreae and Bandelt showed\footnote{See Theorem 2 in \cite{Andreae1995}.} that the application of their method to a minimal spanning tree (and an arbitrary edge) of a complete, weighted graph $(K_n,\omega)$ yields a Hamiltonian cycle $H_{AB}$ which satisfies
\begin{gather}\label{omegaHAB}
\omega(H_{AB}) \leqslant \frac{3\beta^2 + \beta}{2}\cdot \omega(H_{ideal}),
\end{gather}

\noindent
where $H_{ideal}$ is an ideal solution (a minimal Hamiltonian cycle) to the travelling salesperson problem. A couple years later, Andrea and Bandelt reanalysed their method and came up with a way of enhancing its performance.\footnote{For the original paper of Andreae and Bandelt see \cite{Andreae2001}.} For the most part, the \textit{refined Andreae-Bandelt algorithm} (or \textit{rAB algorithm} for short) follows the same steps as its ``basic'' counterpart. The only difference lies in the choice of vertices $y_i,$ so \textbf{Step 2.} is replaced with
\begin{description}
    \item[\textbf{(refined) Step 2.}] If $|V(T_i)|\geqslant 2$ pick $y_i\in E(T_i)$ such that $\{x_i,y_i\} \in E(T_i)$ and
    $$\omega(\{x_i,y_i\}) = \min\bigg\{ \omega(\{x_i,y\})\ :\ \{x_i,y\} \in E(T_i) \bigg\}.$$
    
    \noindent
    If $|V(T_i)| = 1$ put $y_i := x_i.$
\end{description}

\noindent
If $H_{rAB}$ denotes the Hamiltonian cycle constructed by the refined Andreae-Bandelt method, then it turns out\footnote{See Theorem 1 in \cite{Andreae2001}.} that the following estimate holds true:
\begin{gather}\label{omegaHrAB}
\omega(H_{rAB}) \leqslant \frac{\beta^2 + \beta}{2}\cdot \omega(H_{ideal}).
\end{gather}

\subsection{Path matching Christofides algorithm}

The idea of replacing the matching in the original Christofides algorithm by minimum-weight perfect path matching is due to B\"{o}ckenhauer et al. \cite{Bockenhauer2002}. This approach led to a  procedure known as the \textit{path matching Christofides algorithm} (or \textit{PMCh} algorithm for short), which was then slightly refined by Krug \cite{Krug2013} as the original version of the algorithm failed (in certain cases) to deliver a Hamiltonian cycle! The steps for this refined PMCh method are as follows:\\

\noindent
\textbf{Input:} A complete, weighted graph $G:= (K_n,\omega).$

\begin{description}
        \item[\textbf{Step 1.}] Find a minimal spanning tree $T$ of $G$.\footnote{We refer the Reader to \cite[Chapter 3.10]{Deo1974}, \cite[Chapter 12]{Papadimitriou1998} for the descriptions of Kruskal's and Prim's algorithm and \cite{Nesetril2001} for a thorough exposition of Boruvka's algorithm.}
    
        \item[\textbf{Step 2.}] Let $V_{odd}(T)$ be the set of all odd vertices (i.e., vertices with odd degrees) of $T$. Let $F$ be a weighted subgraph induced on $G$ by $V_{odd}(T)$.
    
        \item[\textbf{Step 3.}] Find a minimum-weight perfect path matching\footnote{The path matching $M$ is a family of paths in $G$ with disjoint endpoints which connect pairs of nodes from $F$. The minimum-weight path matching is a path matching with the least total weight. It is said to be perfect if every node in $F$ is an endpoint for one of the paths in $M$. Due to the minimality of this structure, the paths in $M$ form a forest and every pair of paths is edge-disjoint (see \cite[Claim 1]{Bockenhauer2002}).} in $F$ and denote it by $M$. 
        \begin{description}
            \item[\textbf{Step 3.1.}] Compute the shortest paths connecting vertices of $F$ in the original graph $G$.\footnote{Our implementation uses Floyd-Warshall method (see \cite{Floyd1962}) as it is well-suited for ``dense'' graphs.}
            
            \item[\textbf{Step 3.2.}] Construct a complete graph $F'$ on the vertices of $F$ with edge weights equal to the weights of the shortest paths computed in \textbf{Step 3.1}.
            
            \item[\textbf{Step 3.3}] Find a minimum-weight perfect matching $M$ in $F'$.\footnote{Such a matching exists because $F$ has even number of vertices due to the ``handshaking lemma'' (see Proposition 1.2.1 in \cite{Diestel2000}).}
            
            \item[\textbf{Step 3.4}] Let $\mathcal{P}$ be a family of paths ${P_i}$, where $P_i$ is the shortest path (computed in \textbf{Step 3.1}) connecting the endpoints of $i-$th edge from $M$. $\mathcal{P}$ is the sought minimum-weight perfect path matching.
        \end{description}
    
        \item[\textbf{Step 4.}] Resolve conflicts on $\mathcal{P}$ to obtain vertex-disjoint path matching $\mathcal{P}'$. This can be done by finding a path with only one conflict\footnote{``\textit{Conflict}'' is defined as a vertex belonging to more than one path in $\mathcal{P}$. As long as $\mathcal{P}$ is not vertex-disjoint, there always exists a path with exactly one conflict, since the graph composed of all paths in $\mathcal{P}$ is cycle-free.} in $\mathcal{P}$, then either bypassing the conflicting vertex if it is internal to this path or recombining it with the other conflicting path. \footnote{A detailed graphic description of this procedure can be found in \cite[Procedure 1, Fig. 4]{Bockenhauer2002} and in \cite[Algorithm 2]{Krug2013}.}
        
        \item[\textbf{Step 5.}] Construct an Eulerian walk $\eta$ on a multigraph obtained from combining $T$ with the paths from $\mathcal{P},$ which alternates between complete paths from $\mathcal{P}$ and the paths which are subgraphs of $T$. This Eulerian walk can be found in analogous way to the one presented in \textbf{Step 4} of the polygonal Christofides algorithm (see the next subsection) -- simply replace paths from $\mathcal{P}$ with single edges (connecting the endpoints of each path) and apply the enhanced Hierholzer procedure to such a multigraph. Let $\mathcal{Q}$ denote the set of all paths in $\eta$ which were constructed in this part of the procedure as the subpaths of $T$.

        \item[\textbf{Step 6.}] Transform $\mathcal{Q}$ to obtain a forest of degree at most $3$ as follows:
    
        \begin{description}
        \item[\textbf{Step 6.1}] Fix any root vertex $r\in T$. For every vertex $x$ in $T$, let $d_r(x)$ denote its node-distance to the selected root (it can be calculated easily by standard depth-first search).
        
        \item[\textbf{Step 6.2}] For every path $Q\in \mathcal{Q}$ let $x_Q$ be the vertex in $Q$ with the minimal value of $d_r$. If $x_Q$ is not an endpoint of $Q$ and its degree in $T$ exceeds $3$,\footnote{This condition imposed on the degree of $x_Q$ in $T$ is precisely the remedy which was introduced by Krug in \cite{Krug2013}.} redefine $Q$ by omiting the vertex $x_Q$.\footnote{Notice that $Q$ is still a path in $G$ (however, it no longer is a subgraph of $T$) which connects the same endpoints as previously.}
        \end{description}
        
        \item[\textbf{Step 7.}] Replace paths from $T$ which were used in $\eta$ with their refined versions obtained in \textbf{Step 6}. Remove the remaining conflicts as follows:
        \begin{description}
            \item[\textbf{Step 7.1}] Let $x\in \eta$ be any vertex appearing in $\eta$ twice.\footnote{A single vertex can appear in $\eta$ either once or twice.} If its neighbours in the initial state of $\eta$ are conflicts, bypass one of them, otherwise -- bypass any other duplicated vertex. 
            \item[\textbf{Step 7.2}] Repeat \textbf{Step 7.1} until no vertex appears in $\eta$ more than once.
        \end{description}
        \item[\textbf{Step 8.}] Return the refined $\eta$ from \textbf{Step 7.} as the Hamiltonian cycle $H_{PMCh}$. 
    \end{description}

It can be proved\footnote{See \cite[Theorem 2.1]{Krug2013} and \cite[Claim 8.]{Bockenhauer2002}.} that:

\begin{gather}\label{omegaPMCh}
\omega(H_{PMCh}) \leqslant \frac{3\beta^2}{2}\cdot \omega(H_{ideal}). 
\end{gather}

\subsection{Polygonal Christofides algorithm}

The final approximation algorithm we take into consideration is the \textit{polygonal Christofides algorithm} (or \textit{PCh algorithm} for short), which we introduced in our previous  work.\footnote{See \cite{Krukowski2021}.} Using the $\gamma-$polygon structure of the graph, we made necessary adjustments to the classical Christofides algorithm accounting for the fact that the graph need not be metric (i.e., $\gamma$ need not be equal to $1$):\\

\noindent
\textbf{Input:} A complete, weighted graph $G:=(K_n,\omega).$

\begin{description}
        \item[\textbf{Step 1.}] Find a minimal spanning tree $T$ of $G$.\footnote{We refer the Reader to \cite[Chapter 3.10]{Deo1974}, \cite[Chapter 12]{Papadimitriou1998} for descriptions of Kruskal's and Prim's algorithm and \cite{Nesetril2001} for a thorough exposition of Boruvka's algorithm.}
    
        \item[\textbf{Step 2.}] Let $V_{odd}(T)$ be the set of all odd vertices (i.e., vertices with odd degrees) of $T$. Let $F$ be a weighted subgraph induced on $G$ by $V_{odd}(T)$.
    
        \item[\textbf{Step 3.}] Find a minimum-weight perfect matching\footnote{Such matching exists because $F$ is a complete graph with even number of vertices. The latter observation is due to the ``handshaking lemma'' -- see Proposition 1.2.1 in \cite{Diestel2000}.} in $F$ and denote it by $M$. This can be done by applying the original blossom algorithm (due to Edmonds) or one of its subsequent versions.\footnote{The initial version of minimum-weight perfect matching algorithm \cite{Edmonds1965} had complexity of order $\mathcal{O} \left(|E|\cdot |V|^2\right)$. This bound has been consistently improved over the years -- see Tables I and II in \cite{Cook1999} for a detailed exposition.}
    
        \item[\textbf{Step 4.}] Use the following steps (from 4.1 to 4.3) to perform the enhanced Hierholzer algorithm and find an Eulerian walk\footnote{Just as Hamiltonian cycle is a cycle which visits every vertex exactly once, the \textit{Eulerian walk} is a walk which traverses through each edge of the graph exactly once. This walk can be found using either Fleury's algorithm or Hierholzer's algorithm (with time-complexities  $\mathcal{O}(|E|^2)$ and $\mathcal{O}(|E|),$ respectively). These algorithms can be found as X.2 and X.4 in \cite[Chapter 10]{Fleischner1991}, respectively.} $W$ in the multigraph $\GG = (V(T), \EE),$ where 
        $$\EE := \bigg\{(0,e) \ :\ e\in E(T)\bigg\} \cup \bigg\{(1,e) \ :\ e\in E(M)\bigg\}.$$
        
        \begin{description}
            \item[\textbf{Step 4.1}] Select an arbitrary starting vertex $x_1\in V(\mathbb{G})$, which is incident to an edge in $E(M)$. Let $W:=(y_1)$, where $y_1:=x_1$. Mark $y_1$ as a ``\textit{recently visited vertex}''. 
            \item[\textbf{Step 4.2}] Extend the sequence of vertices $W$ in the following way: 
            \begin{description}
                \item[(a)] select any neighbour $y$ of the ``\textit{recently visited vertex}'' connected to it by an ``\textit{unused}'' edge in $E(M)$ (if possible) or in $E(\mathbb{G})$ (this is always possible because each vertex in $V(\mathbb{G})$ has even degree).\footnote{Due to this ``prioritization'' $\{y_1,y\}$ is guaranteed to be taken from $E(M)$.}
                \item[(b)] add $y$ to the sequence $W$. Mark $y$ as the ``\textit{recently visited vertex}'' and denote the traversed edge as ``\textit{used}''. If there is an unused edge incident to $y$, go back to step a).\footnote{At each step of the algorithm the only vertices with odd number of unused edges are $x_1$ and the current vertex $y$. Therefore, the loop can terminate only if we return to the initial vertex $x_1$.}
            \end{description}
            \item[\textbf{Step 4.3}] If there are any unused edges after this process, start at any vertex $x\in W$ which has at least one neighbour not in $W$. Repeat the procedure described in Step 4.2 obtaining a closed walk $W_x.$ Replace the last appearance of $x$ in sequence $W$ with $W_x$.
        \end{description}
        \item[\textbf{Step 5.}] Use the following steps (5.1 and 5.2) to perform the enhanced shortcutting procedure and obtain a Hamiltonian cycle $H_{PCh} = (y_1,y_2,\ldots,y_n,y_{n+1})$ on $G$.
        
        \begin{description}
            \item[\textbf{Step 5.1}] Put $y_1:=x_1$ and $y_2:=x_2$. From the construction (in Step 4.) of the Eulerian walk $W$ it follows that $\{x_1,x_{2}\}\in E(M)$. Let $i:=3$ and $j:=3$.
            \item[\textbf{Step 5.2}] While $i\leqslant n$ perform the following steps:
                \begin{description}
                    \item[(a)] If $x_j\notin V(M)$ and it has already appeared in $H_{PCh}$, increment $j$. 
            
                    \item[(b)] If $x_j\notin V(M)$ and it has not yet appeared in $H_{PCh}$, put $y_i:=x_j$ and increment both $i$ and $j$.
            
                    \item[(c)] If $x_j\in V(M)$, $\{x_j,x_{j+1}\}\notin E(M)$ and $\{x_{j-1},x_j\}\notin E(M)$, increment $j$. 
            
                    \item[(d)] Otherwise, i.e., in the situation where $x_j\in V(M)$ and either $\{x_j,x_{j+1}\}\in E(M)$ or $\{x_{j-1},x_j\}\in E(M)$, let $y_i:=x_j$ and increment both $i$ and $j$.
                \end{description}
        \end{description}
\end{description}

The climax of our previous paper was the proof that the PCh algorithm produces a Hamiltonian cycle $H_{PCh},$ which satisfies the following estimate:\footnote{See Theorem 8 in \cite{Krukowski2021}.}
\begin{gather}\label{omegaHpolygon}
\omega(H_{PCh}) \leqslant \frac{3\gamma}{2}\cdot \omega(H_{ideal}).
\end{gather}

As a closing remark of this section we may compare the estimate \eqref{omegaHpolygon} with those of the previous algorithms (see \eqref{omegaHDMST}, \eqref{omegaHAB}, \eqref{omegaHrAB}, \eqref{omegaPMCh}) and arrive at the conclusion that the PCh algorithm flaunts the best worst-case behaviour of all the approximation methods (whenever $\gamma\in [\beta,2\beta)$ and $\beta \geqslant 3$). Next section is devoted to supporting this claim on the basis of numerical simulations.

\section{Numerical comparison of the approximation methods}
\label{section:numericalcomparison}

In the previous section we reviewed a number of algorithms generating approximate solutions to the travelling salesperson problem. Apart from the methods themselves, we have provided estimates \eqref{omegaHDMST}, \eqref{omegaHAB}, \eqref{omegaHrAB}, \eqref{omegaPMCh} and \eqref{omegaHpolygon}, which tell us how far a total weight of an approximate solution can deviate from the total weight of an ideal Hamiltonian cycle (i.e., a ``true'' solution to the TSP). For instance, \eqref{omegaHDMST} guarantees that the weight of the approximate solution generated by the DMST method cannot exceed $2\gamma$ times the weight of the ideal Hamiltonian cycle, whereas the application of the PCh algorithm reduces this constant to $\frac{3\gamma}{2}.$ These estimates provide valueable insights into worst-case scenarios of the algorithms' performance but this theoretical deliberations are far from being the sole way of measuring the quality of the presented methods. In the current section, instead of focusing on the worst-case scenarios, we concentrate on examples and numerical simulations, which test how the algorithms fare in ``real life''. 

We commence with a concrete example of a weighted $K_7$ graph, whose nodes are labeled ``$0$'', ``$1$'', $\ldots$, ``$6$'' (see Fig. \ref{fig:exampleofK7}). Browsing through every Hamiltonian cycle (there are 360 of them) or running the Held-Karp algorithm we discover the solution $(0,1,2,4,5,3,6)$ to the TSP with the total weight of $2.07$ (see Fig. \ref{fig:K7_and_solution_to_TSP}). The following table presents the performance of approximation algorithms reviewed in the previous section:

\begin{center}
\begin{tabular}{ | c | c | c |}
    \hline
     Algorithm & Hamiltonian cycle & Total weight \\ \hline
     DMST & $(0,1,2,6,3,5,4)$ & 2.22 \\ \hline
     AB & $(0,2,1,3,5,6,4)$ & 3.29 \\ \hline
     rAB & $(0,1,2,6,5,3,4)$ & 3.08 \\ \hline
     PMCh & $(0,1,2,6,3,5,4)$ & 2.22 \\ \hline
     PCh & $(0,4,2,1,6,3,5)$ & 2.18 \\ \hline
\end{tabular}
\end{center}

The cycles obtained by each of the algorithms are presented in the following figures:

\begin{figure}[H]
        \hspace{-0.075\textwidth}\includegraphics[width=1.15\textwidth]{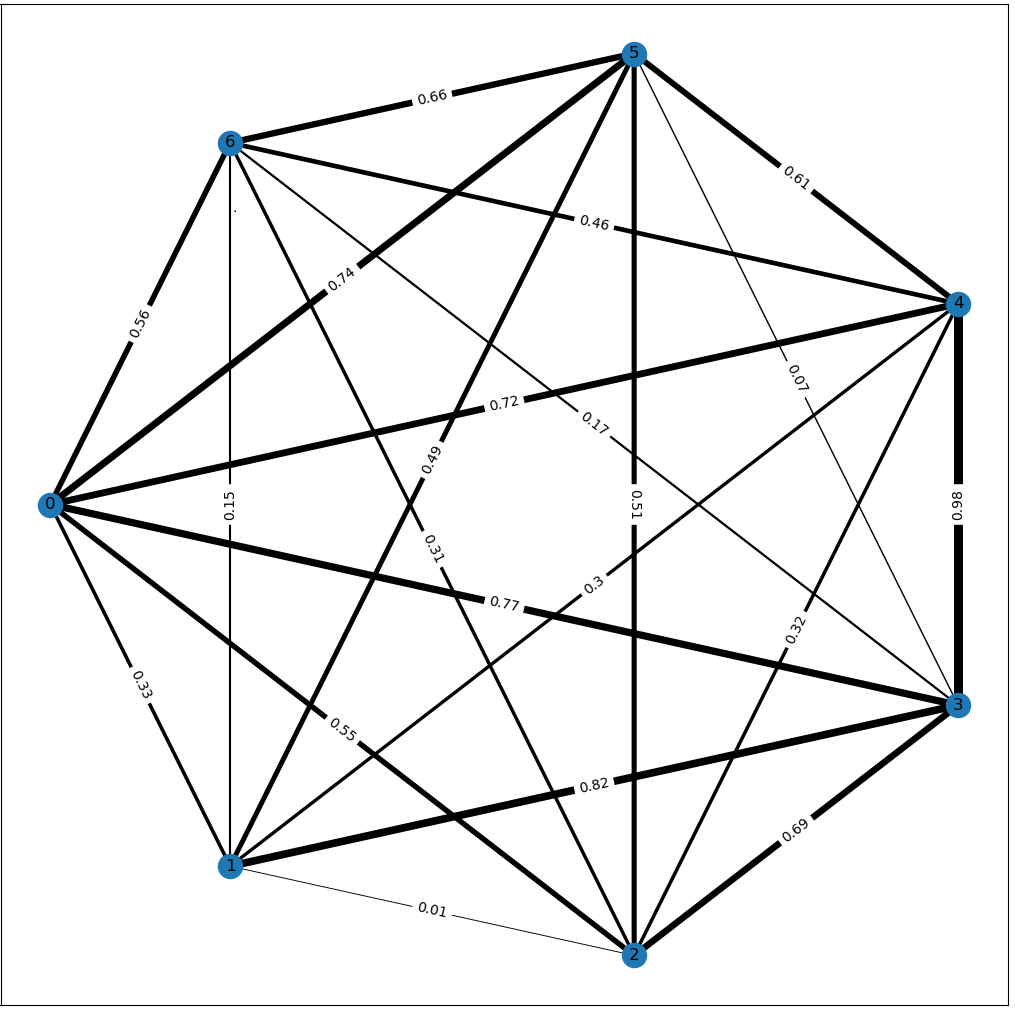}
				\vspace{-.5cm}
        \caption{Example of a weighted $K_7$ graph.}
        \label{fig:exampleofK7}
\end{figure}

\begin{figure}[H]
        \hspace{-0.075\textwidth}\includegraphics[width=1.15\textwidth]{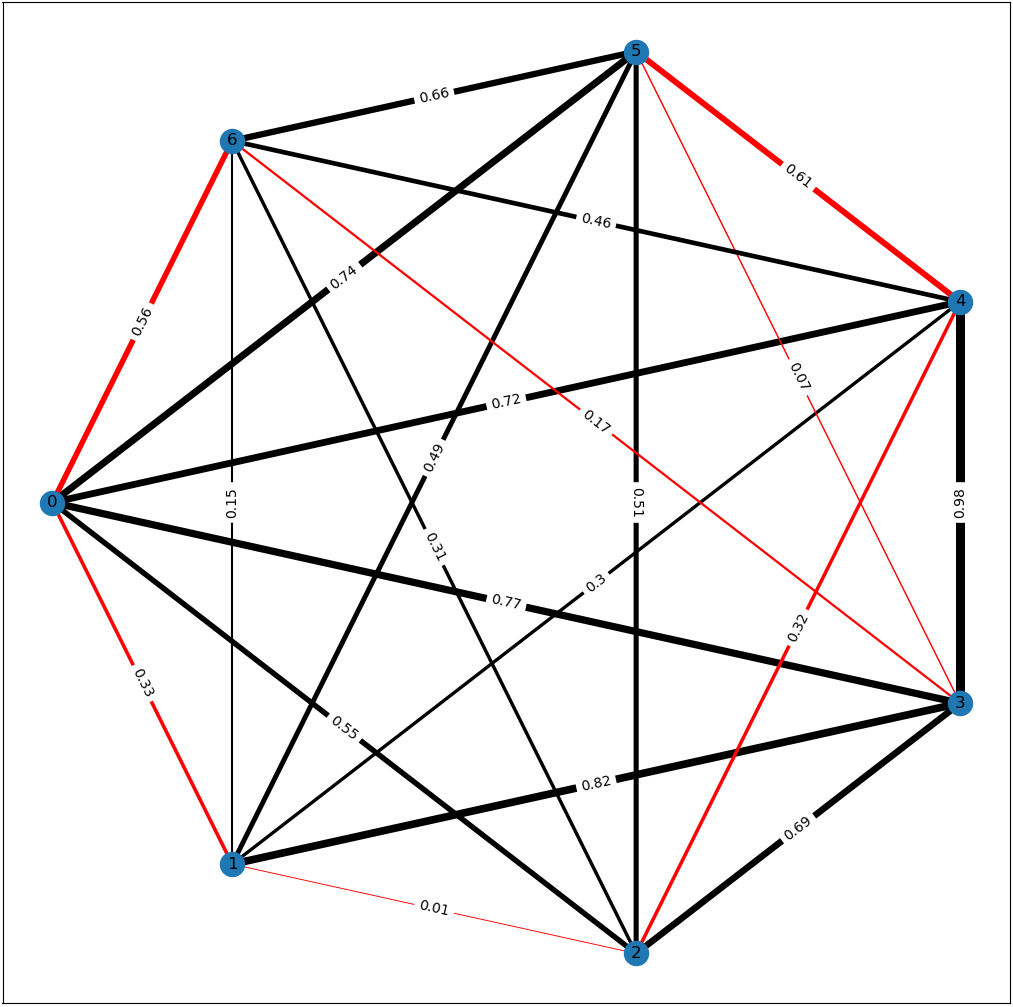}
				\vspace{-.5cm}
        \caption{Ideal solution to the TSP (in red).}
        \label{fig:K7_and_solution_to_TSP}
\end{figure}

		\begin{figure}[H]
        \hspace{-0.075\textwidth}\includegraphics[width=1.15\textwidth]{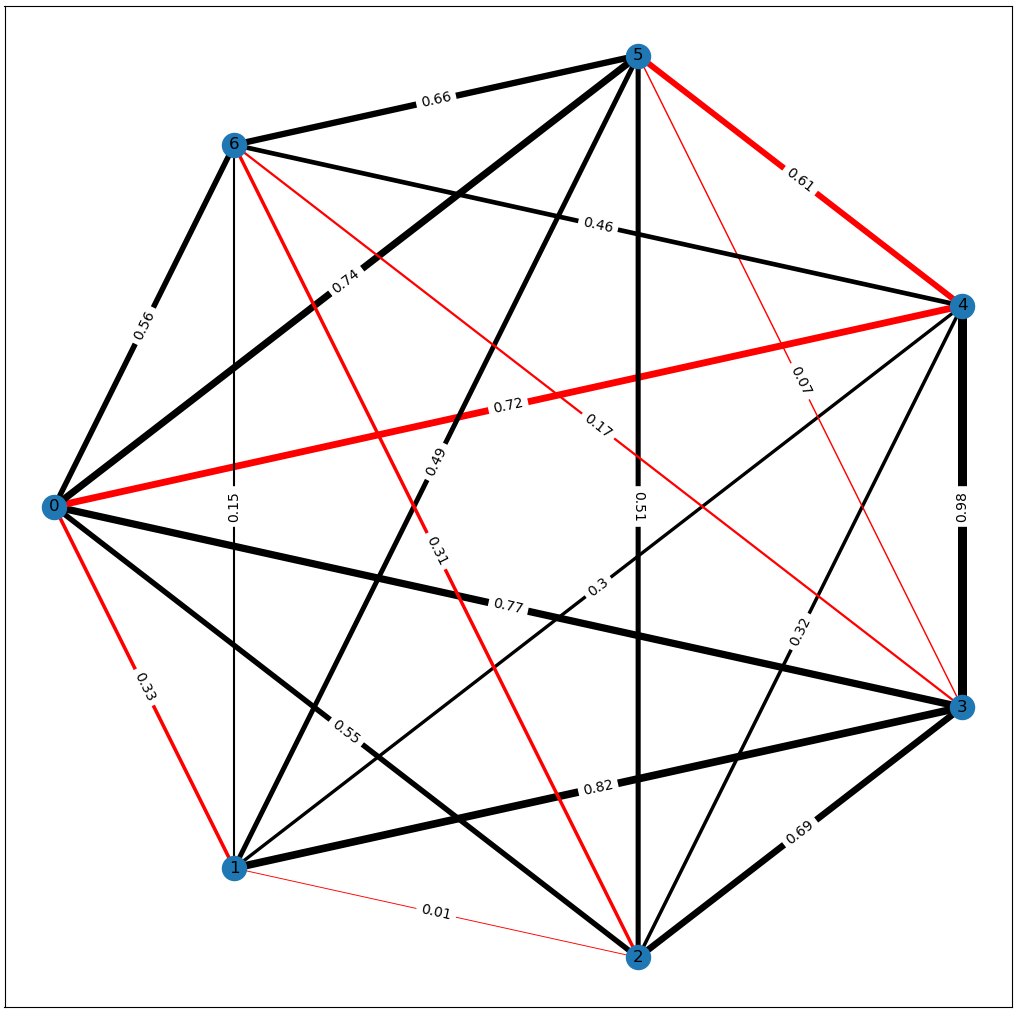}
				\vspace{-.5cm}
        \caption{Approximate solution generated by the DMST or the PMCh algorithm (in red).}
        \label{fig:K7_and_DMST}
    \end{figure}

    \begin{figure}[H]
        \hspace{-0.075\textwidth}\includegraphics[width=1.15\textwidth]{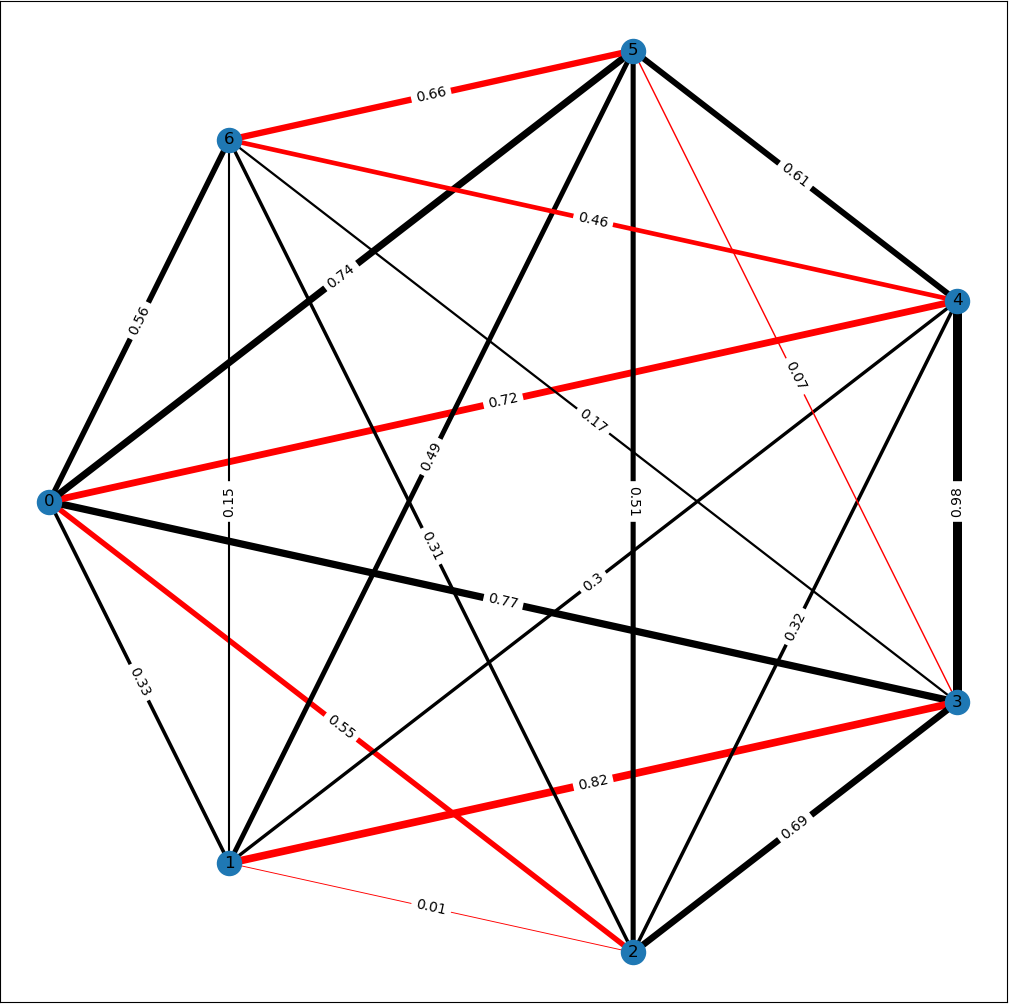}
				\vspace{-.5cm}
        \caption{Approximate solution generated by the AB algorithm (in red).}
        \label{fig:K7_and_A}
    \end{figure}
    
		\begin{figure}[H]
        \hspace{-0.075\textwidth}\includegraphics[width=1.15\textwidth]{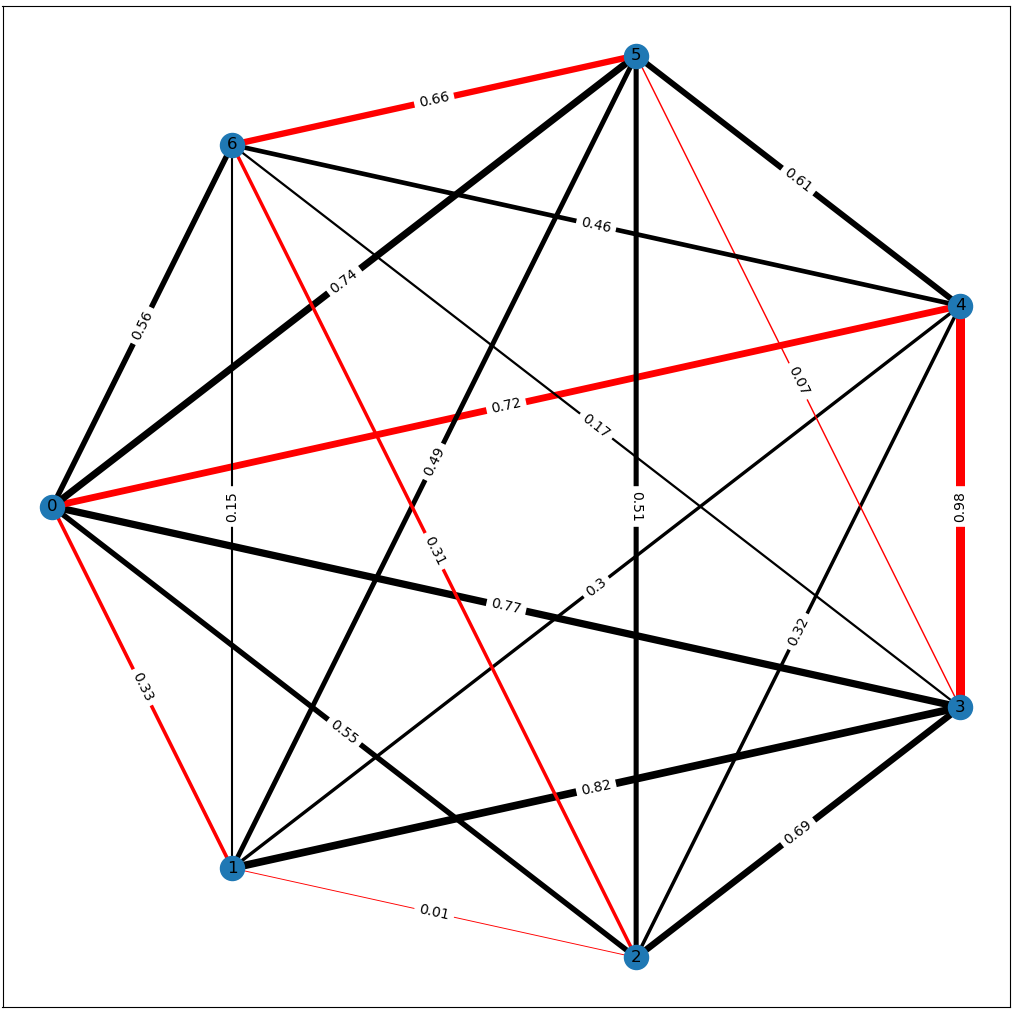}
				\vspace{-.5cm}
        \caption{Approximate solution generated by the rAB algorithm (in red).}
        \label{fig:K7_and_rA}
    \end{figure}
    
		\begin{figure}[H]
        \hspace{-0.075\textwidth}\includegraphics[width=1.15\textwidth]{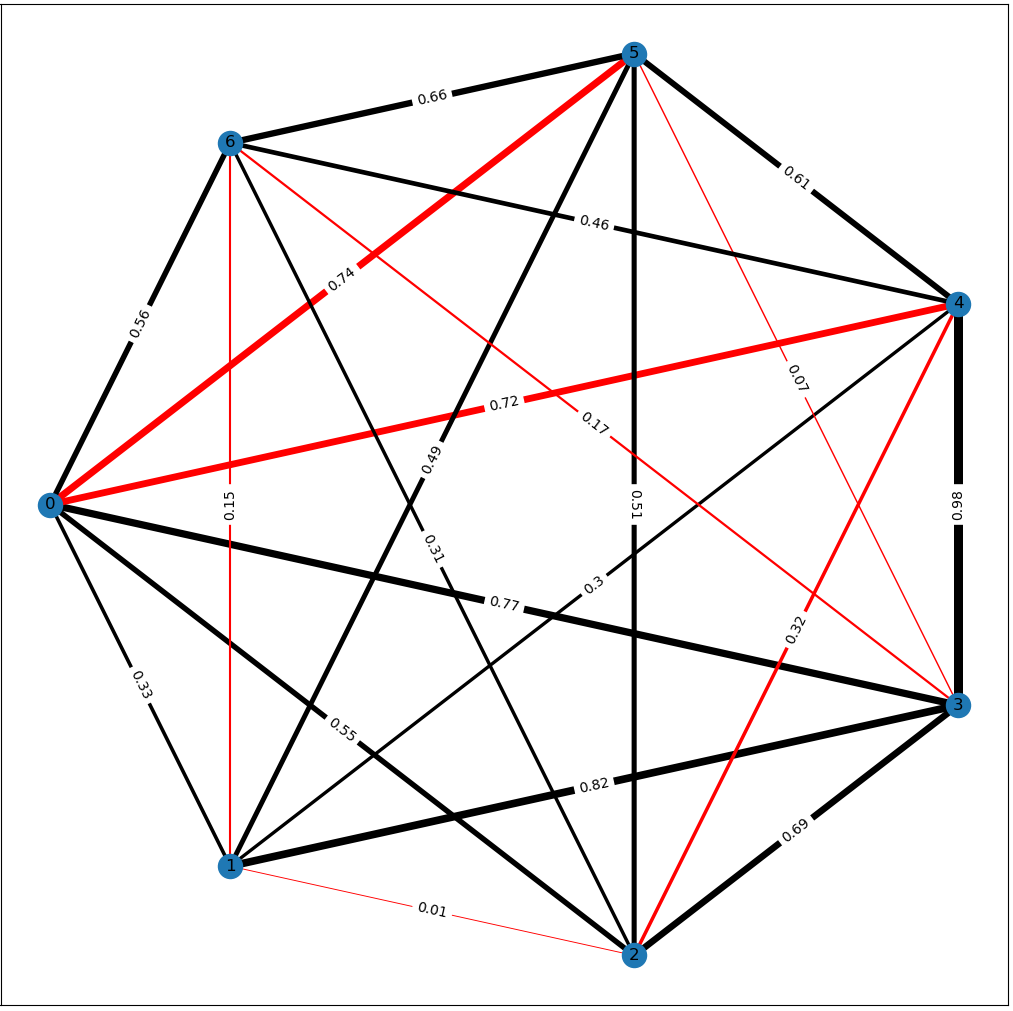}
				\vspace{-.5cm}
        \caption{Approximate solution generated by the PCh algorithm (in red).}
        \label{fig:K7_and_KT}
    \end{figure}

Although this preliminary instance seems promising, we should not jump to any conclusions on the basis of a solitary example. In order to avoid accusations that the $K_7$ example given above was a ``fluke'', we have devised the following experiment: we randomize 45 graphs with 75 nodes and run all 5 algorithms to find the approximate solutions on these graphs. The results are enclosed below:

\vspace{-.1cm}
\begin{figure}[H]
    \centering
    \includegraphics[width=0.9\textwidth]{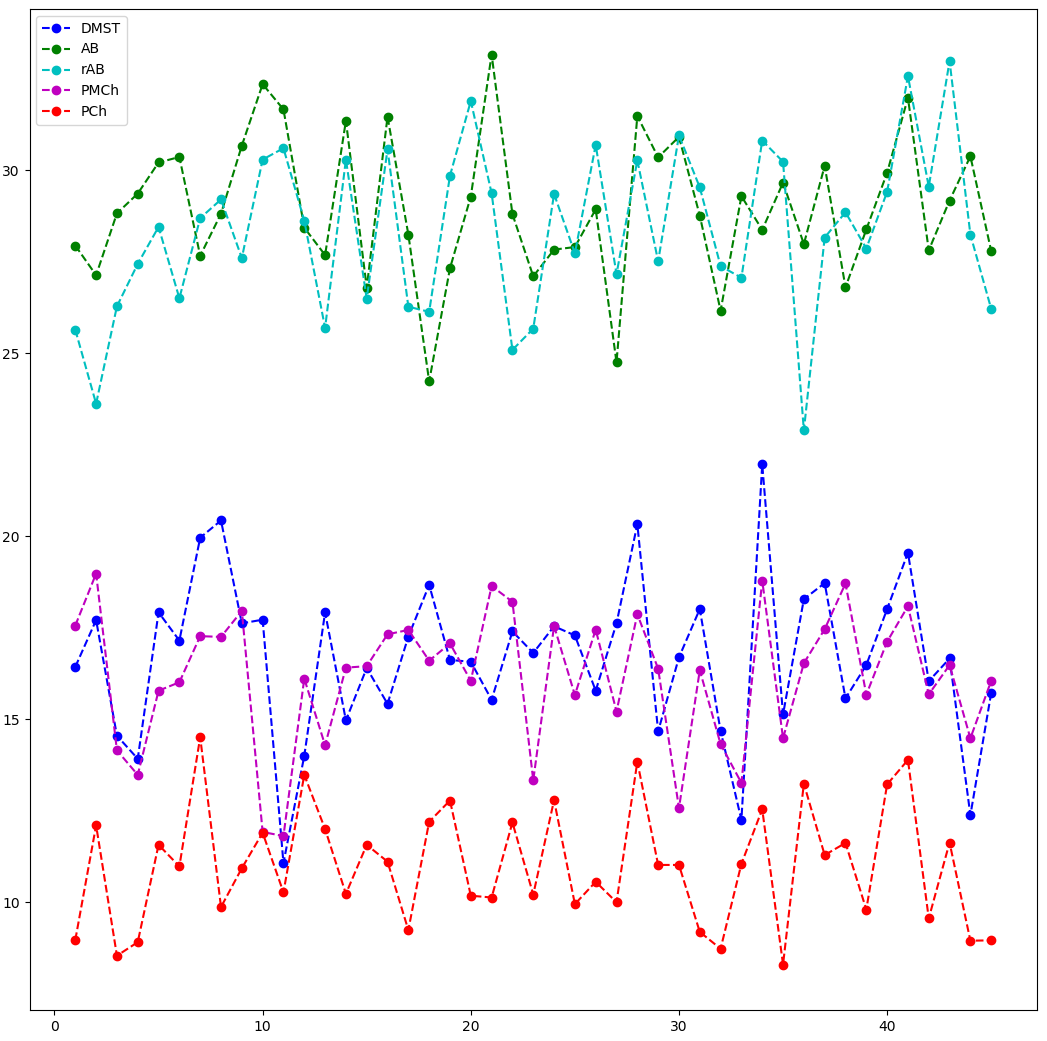}
		\vspace{-.5cm}
    \caption{Approximate solutions to the TSP on 45 randomized graphs with 75 nodes. X-axis represents the number of the ``test'', Y-axis represents the total weight of the Hamiltonian cycle.}
    \label{fig:multiplot75}
\end{figure}

We can distinguish 3 ``layers'' in this plot. The first one spans roughly from 25 to 35 -- this interval contains most of the approximate solutions generated by AB and rAB algorithms. The second ``layer'' is from 13 or 14 to 20 and contains the bulk of approximate solutions returned by the DMST and PMCh methods. The third and final layer consists of a single (red) plot representing the approximate solutions produced by the PCh algorithm. It is clear that these approximate solutions are better (i.e., have lower total weight) than the ones generated by all other methods. 

In order to confirm the conclusions drawn from the first experiment, we have run it again, increasing the size of the graphs to 100 and then to 125 nodes. As seen in the figure below, the dominance of PCh algorithm over all other methods remains unquestioned.

\vspace{-.1cm}
\begin{figure}[H]
    \centering
        \includegraphics[width=0.9\textwidth]{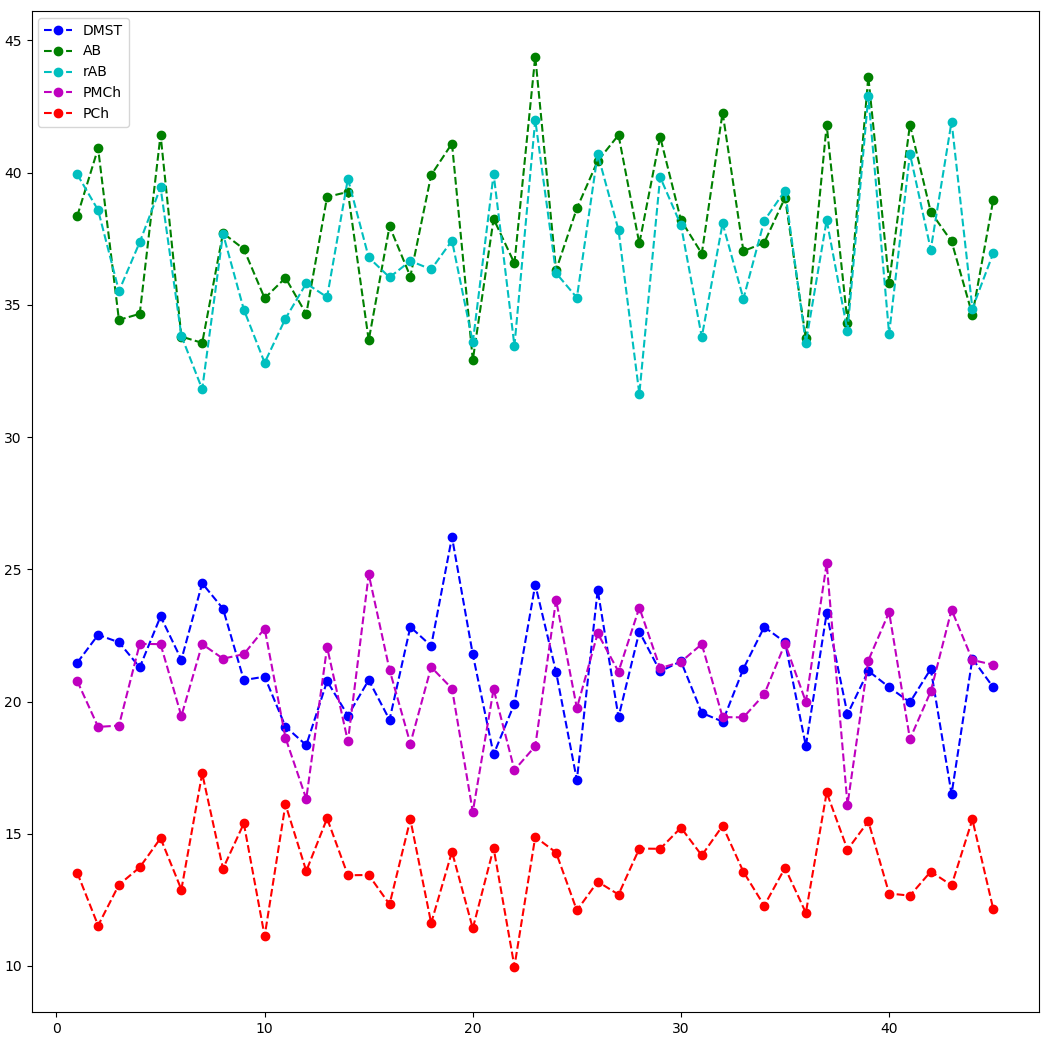}
        \label{fig:multiplo100}
				\vspace{-.5cm}
				 \caption{Approximate solutions to the TSP on 45 randomized graphs with 100 nodes. X-axis represents the number of the ``test'', Y-axis represents the total weight of the Hamiltonian cycle.}
\end{figure}
 
\begin{figure}[H]
        \centering
        \includegraphics[width=0.9\textwidth]{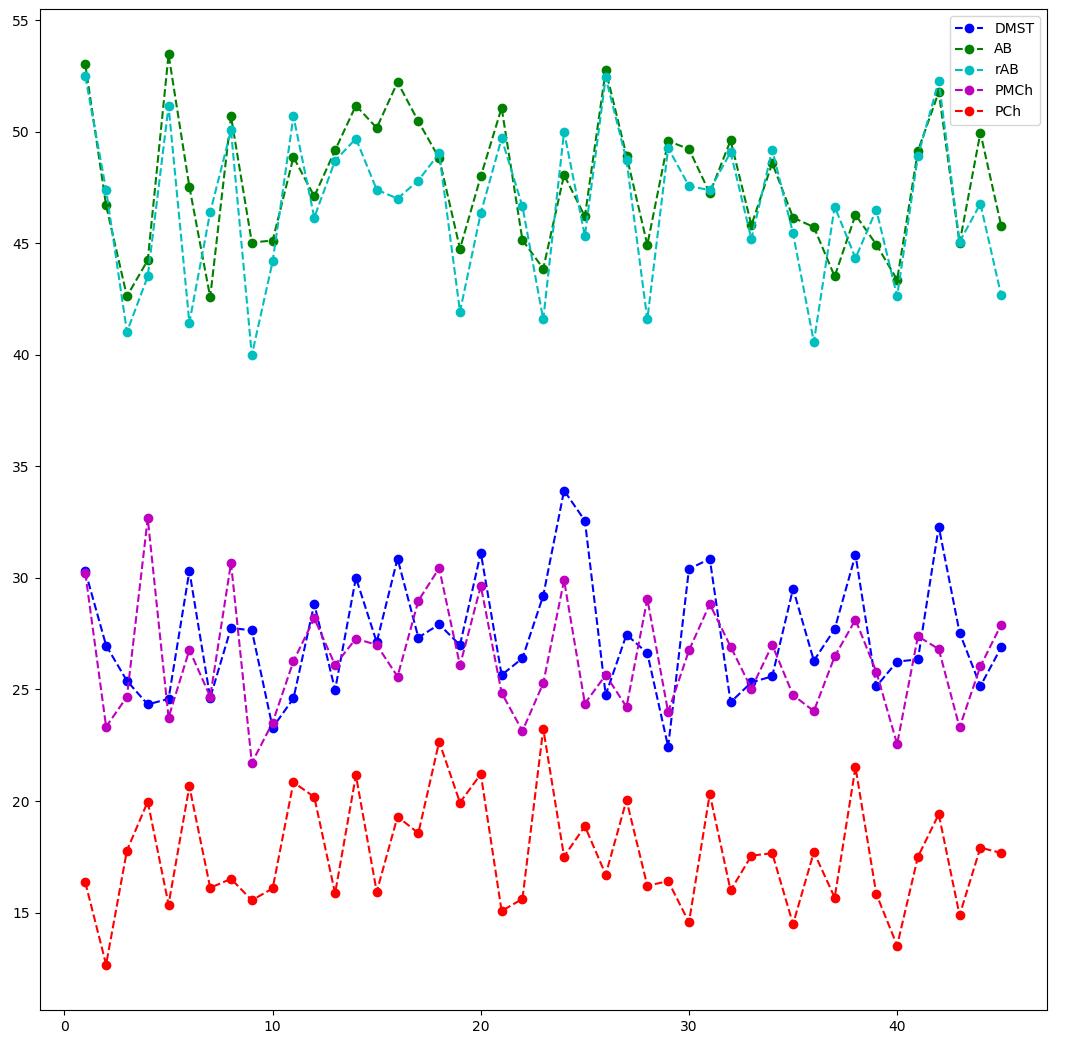}
        \label{fig:multiplot125}
				\vspace{-.5cm}
    \caption{Approximate solutions to the TSP on 45 randomized graphs with 125 nodes. X-axis represents the number of the ``test'', Y-axis represents the total weight of the Hamiltonian cycle.}
\end{figure}

At this point we should be fairly convinced that the PCh algorithm returns better approximate solutions for larger graphs than the rest of the discussed methods. One last doubt that should be dispelled is that of time complexity. After all, the Held-Karp algorithm returns the ideal solution to the TSP, but as we have remarked earlier, its exponential time complexity makes it impractical. In theory, the PCh method should not suffer from this cardinal flaw since its time complexity equals $\mathcal{O}(n^3)$ (see Theorem 9 in \cite{Krukowski2021}). Let us verify this claim with the following numerical simulation: we randomize a 100 graphs of size $n$ for every $n = 5,\ldots, 100$ and on every chunk of these 100 graphs we run all 5 algorithms and compute the average time of execution. The results are illustrated in the picture below:

\vspace{-0.1cm}
\begin{figure}[H]
    \centering
    \includegraphics[width=0.9\textwidth]{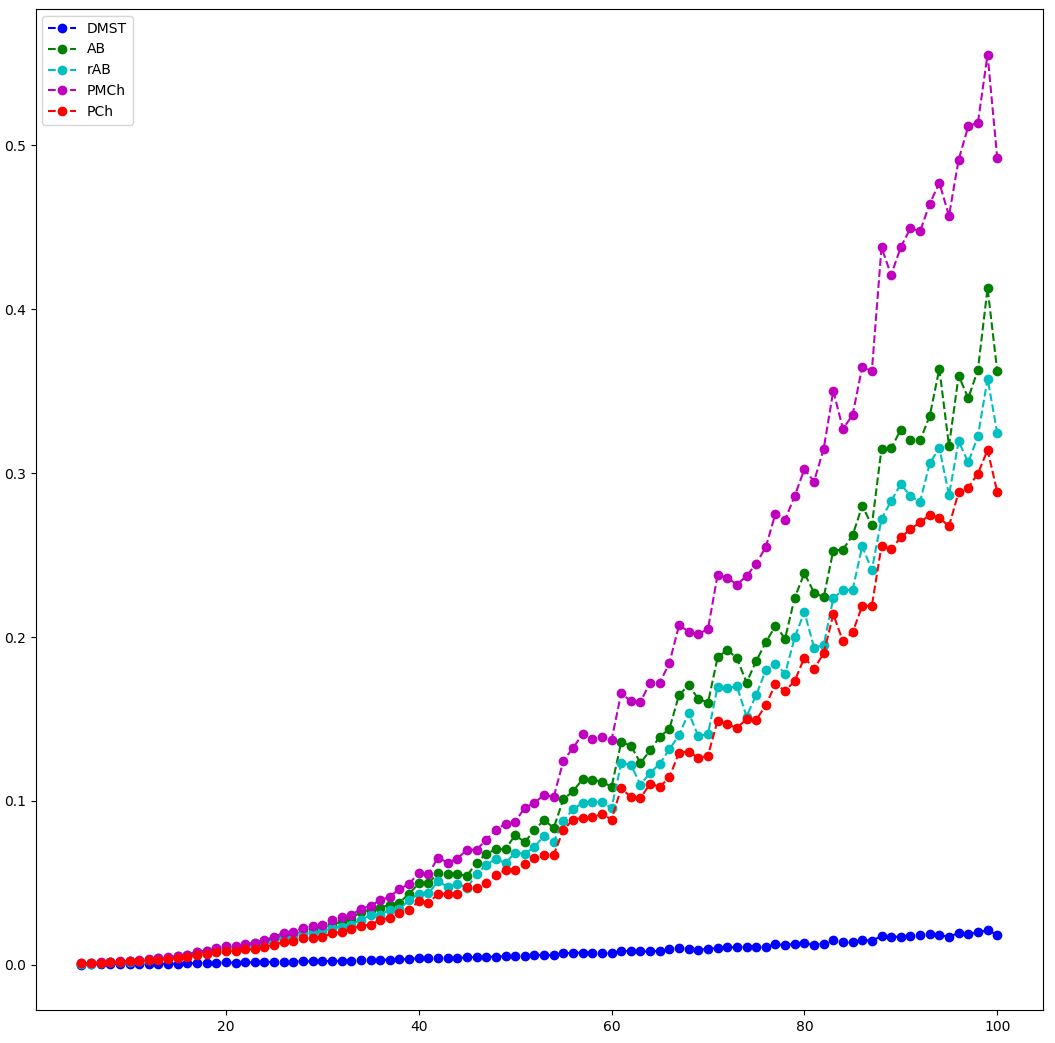}
    \caption{Average time of execution of approximate algorithms on chunks of 100 randomized graphs. X-axis represents the number of the chunk and the size of graphs at the same time. Y-axis represents time in seconds.}
    \label{fig:timecomplexity}
\end{figure}

One thing that is impossible to overlook in this plot is the excellent time-efficiency of the DMST method. This complies with the fact that in theory this algorithm has the time complexity of $\mathcal{O}(n^2\log(n))$. The next thing that catches one's eye is the fact that the PCh algorithm is more time efficient than the PMCh and both variants of the Andreae-Bandelt method. 

To conclude, we have strong evidence that our PCh method is one of the best approximation algorithms for solving the TSP. It produces Hamiltonian cycles of lower weights than the rest of the methods. Furthermore, its average execution time is comparable with those of other algorithms (bar the DMST method). It is our firm belief that these features position the PCh method as one of the best approximation algorithms for solvng TSP currently known in the literature.

\end{document}